\documentclass[1p]{elsarticle}

\usepackage{lineno}
\modulolinenumbers[5]

\journal{ }

\usepackage{mathrsfs}
\usepackage{amsfonts}
\usepackage{enumerate}
\usepackage{latexsym}
\usepackage[all,cmtip]{xy}
\biboptions{numbers,sort&compress}

\usepackage{verbatim}

\usepackage[colorlinks,linkcolor=blue,citecolor=blue]{hyperref}

\usepackage{amsmath,amssymb,xspace,amsthm}
\usepackage{bbm}

\numberwithin{equation}{section}

\def\fg{\mathfrak{g}}

\def\fk{\mathfrak{k}}
\def\fm{\mathfrak{m}}

\def\fB{\mathfrak{B}}

\def\fK{\mathfrak{K}}
\def\fL{\mathfrak{L}}

\def\cA{\mathcal{A}}

\def\cF{\mathcal{F}}

\def\cK{\mathcal{K}}

\def\cT{\mathcal{T}}

\def\sJ{\mathscr{J}}

\def\bC{\mathbb{C}}
\def\bN{\mathbb{N}}
\def\bZ{\mathbb{Z}}

\def\Supp{\mathrm{Supp}}

\def\Ker{\mathrm{Ker}}

\def\ptl{\partial}

\newtheorem{theo}{{Theorem}}[section]
\newtheorem{lemm}[theo]{Lemma}
\newtheorem{rema}[theo]{Remark}

\newtheorem{coro}[theo]{Corollary}
\newtheorem{prop}[theo]{Proposition}

\allowdisplaybreaks

\begin{document}

\begin{frontmatter}



\title{Classification of simple Harish-Chandra modules over the Neveu-Schwarz algebra and its contact subalgebra}


\author{Yan-an Cai, Rencai L\"{u}}

\begin{abstract}
In this paper, we classify all simple jet modules for the Neveu-Schwarz algebra $\widehat{\fk}$ and its contact subalgebra $\fk^+$. Based on these results, we give a classification of simple Harish-Chandra modules for $\widehat{\fk}$ and $\fk^+$.
\end{abstract}


\end{frontmatter}


\section{Introduction}

We denote by $\bZ, \bZ_+, \bN, \bC$ and $\bC^*$ the sets of all integers, non-negative integers, positive integers, complex numbers, and nonzero complex numbers, respectively. All vector spaces and algebras in this paper are over $\bC$. We denote by $U(L)$ the universal enveloping algebra of the Lie (super)algebra $L$ over $\bC$. Also, we denote by $\delta_{i,j}$ the Kronecker delta. Throughout this paper, by subalgebras, submodules for Lie superalgebras we mean subsuperalgebras and subsupermodules respectively.

Superconformal algebras have been widely studied in mathematical physics. The simplest examples are the Virasoro algebra, the central extension of the Witt algebra $W$, and its subalgebra $W^+$. An important class of modules for the Virasoro algebra and $W^+$ are the so-called quasifinite modules (or Harish-Chandra modules), the weight modules with finite dimensional weight spaces, which were classified by Mathieu in \cite{Ma}.

After the Virasoro algebra (corresponding to $N=0$), we have the $N=1$ superconformal algebras, also known as the super-Virasoro algebras: the Neveu-Schwarz algebra and the Ramond algebra. Weight modules for the super-Viraoro algebras  have been extensively investigated (cf. \cite{DLM,IK1,IK2}), for more related results we refer the reader to \cite{CP, Rao, Ka, Ka1, KL, KS, LPX1, LPX2, Ma, MZe, MZ, S1} and references therein.

Based on the classification of simple jet modules introduced by Y. Billig in \cite{B} (see also \cite{Rao}), the authors gave a complete classification of simple Harish-Chandra modules for Lie algebra of vector fields on a torus with the so-called $A$ cover theory in \cite{BF}. Recently, with the study of jet modules, a classification of simple Harish-Chandra modules for generalized Heisenberg-Virasoro, the Witt superalgebras, etc., were given, see \cite{LG, LZ, XL} and references therein.

A complete classification for the $N=1$ superconformal algebra was given in \cite{S2}. However, the complicated computations in the proofs maike it difficult to follow. Recently, with the theory of the $A$-cover,
a new approach to classify all simple Harish-Chandra modules for the $N=1$ Ramond algebra $\mathfrak{s}$ was given in \cite{CLL}. However, it is more difficult to classify such modules for the Neveu-Schwarz algebra.

A superversion of $W$ and $W^+$ are the Lie superalgebras  of contact vector fields on the supercircle $S^{1|1}$ and the superspace $\bC^{1|1}$, say $\fk$ (also known as the centerless Neveu-Schwarz algebra) and $\fk^+$.  In this paper, we classify simple jet modules for $\fk^+$. As a consequence, we give a classification of simple jet modules for $\fk$. Then based on these results, with the $A$ cover theory we classify simple Harish-Chandra modules for the Neveu-Schwarz algebra $\widehat{\fk}$ and its contact subalgebra $\fk^+$, which is a super version of Mathieu's results on the Virasoro algebra and its subalgebra $W^+$ in \cite{Ma}. This also gives a new approach for Su's results in \cite{S2}.

The paper is organized as follows. In Section \ref{pre}, we collect some basic results for our study. Simple jet modules are classified in Section \ref{cuspidalAK}. Finally, we classify all simple Harish-Chandra modules for the Neveu-Schwarz algebra and its contract subalgebra in Section \ref{main}.

\section{Preliminaries}
\label{pre}

The Neveu-Schwarz algebra $\widehat{\fk}$ is the Lie superalgebra over $\bC$ with a basis $\{L_n,G_r,C\,|\,n\in\bZ, r\in\bZ+\frac{1}{2}\}$ satisfying the commutation relations
\begin{align*}
&|L_n|=\bar{0}, |G_r|=\bar{1}, |C|=\bar{0};\\
&[L_m,L_n]=(n-m)L_{m+n}+\delta_{m+n,0}\frac{m^3-m}{12}C;\\
&[L_m,G_r]=(r-\frac{m}{2})G_{m+r};\\
&[G_r,G_s]=-2L_{r+s}+\frac{1}{3}\delta_{r+s,0}(r^2-\frac{1}{4})C.
\end{align*}
$\widehat{\fk}$ is a $\frac{1}{2}\bZ$-graded Lie superalgebra with $\widehat{\fk}_i=\bC L_i+\delta_{i,0}\bC C, \forall i\in\bZ$ and $\widehat{\fk}_{j}=\bC G_{j}, \forall j\in\frac{1}{2}+\bZ$. Let $\fk=\widehat{\fk}/\bC C$. Then $\fk$ is also a $\frac{1}{2}\bZ$-graded Lie superalgebra. For $p\in\frac{1}{2}\bZ$, set $\fk_{\geq p}=\sum\limits_{i\geq p}\fk_i$. Then $\fk_{\geq-1}$ and $\fk_{\geq p}(p\in\frac{1}{2}\bZ_+)$ are Lie supersubalgebra of $\fk$. In particular, $\fk^+=\fk_{\geq-1}$ is also a subalgebra of $\widehat{\fk}$, called the contact subalgebra of $\widehat{\fk}$, and $\fk_{\geq p}$ is an ideal of $\fk_{\geq0}$ for any $p\in\frac{1}{2}\bZ_+$. 

Let $A=\bC[t,t^{-1}]\otimes\Lambda(1)$, where $\Lambda(1)$ is the Grassmann algebra in one variable $\xi$. This algebra is $\bZ_2$-graded with $|t|=\bar{0},|\xi|=\bar{1}$. $A$ is naturally a $\fk$-module: for all $x\in A,i\in\bZ, m\in\bZ+\frac{1}{2}$,
\begin{align*}
L_i\circ x&=t^{i+1}\ptl_t(x)+\frac{1}{2}(i+1)t^i\xi\ptl_\xi(x),\\
G_m\circ x&=t^{m+\frac{1}{2}}\xi\ptl_t(x)-t^{m+\frac{1}{2}}\ptl_\xi(x),
\end{align*}
here $\ptl_t=\frac{\ptl}{\ptl t},\ptl_\xi=\frac{\ptl}{\ptl\xi}$. Hence, we have the extended Neveu-Schwarz algebras $\tilde{\fk}=\fk\ltimes A, \tilde{\fk}^+=\fk^+\ltimes A$ and $\tilde{\fk}^{++}=\fk^+\ltimes A^+$, with $A$ and $A^+$ being abelian Lie superalgebras,  where $A^+=\bC[t]\otimes\Lambda(1)$.

On the other hand, $\fk$ ($\fk^+$, respectively) has a natural module structure over the abelian superalgebra $A$ ($A^+$, respectively):
\begin{equation}\label{a-action}
t^iL_j=L_{i+j}, t^iG_m=G_{m+i}, \xi L_j=\frac{1}{2}G_{j+\frac{1}{2}}, \xi G_m=0.
\end{equation}
And $\fk$ ($\fk^+$, respectively) is a $\tilde{\fk}$ ($\tilde{\fk}^{++}$) module with adjoint $\fk$ ($\fk^+$, respectively) actions and $A$ ($A^+$) acting as \eqref{a-action}. To see this, we only need to verify
\[
v(ax)-(-1)^{|v||a|}a(vx)=[v,a]x=(v\circ a)x,
\]
for all homogeneous $x,v\in\fk$ ($\fk^+$, respectively) and $a\in A$ ($A^+$, respectively). In fact, we have
\begin{align*}
&[L_m,t^iL_n]-t^i[L_m,L_n]=[L_m,L_{n+i}]-(n-m)t^iL_{m+n}=iL_{m+n+i}=it^{m+i}L_n,\\
&[L_m,\xi L_n]-\xi[L_m,L_n]=\frac{1}{2}[L_m,G_{n+\frac{1}{2}}]-(n-m)\xi L_{m+n}=\frac{m+1}{4}G_{m+n+\frac{1}{2}}=\frac{m+1}{2}t^m\xi L_n,\\
&[L_m,t^iG_r]-t^i[L_m,G_r]=[L_m,G_{r+i}]-(r-\frac{m}{2})t^iG_{m+r}=iG_{m+r+i}=it^{m+i}G_r,\\
&[L_m,\xi G_r]-\xi[L_m,G_r]=0,\\
&[G_r,t^iL_n]-t^i[G_r,L_n]=[G_r,L_{n+i}]+(r-\frac{n}{2})t^iG_{r+n}=\frac{i}{2}G_{r+n+i}=it^{r+i-\frac{1}{2}}\xi L_n,\\
&[G_r,t^iG_s]-t^i[G_r,G_s]=[G_r,G_{s+i}]+2t^iL_{r+s}=0,\\
&[G_r,\xi L_n]+\xi[G_r,L_n]=\frac{1}{2}[G_r,G_{n+\frac{1}{2}}]-(r-\frac{n}{2})\xi G_{r+n}=-L_{r+n+\frac{1}{2}}=-t^{r+\frac{1}{2}}L_n,\\
&[G_r,\xi G_s]+\xi[G_r,G_s]=-2\xi L_{r+s}=-G_{r+s+\frac{1}{2}}=-t^{r+\frac{1}{2}}G_s.
\end{align*}

An $A\fk$ ($A\fk^+, A^+\fk^+$, respectively) module is a $\widetilde{\fk}$ ($\widetilde{\fk}^+,\widetilde{\fk}^{++}$, respectively) module with $A$ ($A, A^+$, respectively) acting associatively. Let $U=U(\widetilde{K}), U^+=U(\widetilde{K}^+), U^{++}=U(\widetilde{K}^{++})$ and $I$ ($I^+,I^{++}$, respectively) be the left ideal of $U$ ($U^+, U^{++}$, respectively) generated by $t^i\cdot t^j-t^{i+j}, t^0-1,t^i\cdot\xi-t^i\xi$ and $\xi\cdot\xi$ for all $i,j\in\bZ$ ($\bZ, \bZ_{+}$, respectively). Then it is clear that $I$ ($I^+, I^{++}$ respectively) is an ideal of $U$ ($U^+, U^{++}$ respectively). Let $\overline{U}=U/I, \overline{U}^+=U^+/I^+, \overline{U}^{++}=U^{++}/I^{++}$. Then the category of $A\fk$ ($A\fk^+,A^+\fk^+$ respectively) modules is naturally equivalent to the category of $\overline{U}$ ($\overline{U}^+,\overline{U}^{++}$ respectively) modules.

Let $\fg$ be any of $\fk,\fk^+,\widetilde{\fk},\widetilde{\fk}^+,\widetilde{\fk}^{++},\widehat{\fk}, \fk_{\ge0}$. A $\fg$ module $M$ is called a \emph{weight} module if the action of $L_0$ on $M$ is diagonalizable. Let $M$ be a weight $\fg$ module. Then $M=\bigoplus\limits_{\lambda\in\bC}M_\lambda$, where $M_\lambda=\{v\in M\,|\,L_0v=\lambda v\}$, called the weight space of weight $\lambda$. $\Supp(M):=\{\lambda\in\bC\,|\,M_\lambda\neq0\}$ is called the support of $M$. A \emph{cuspidal} (\emph{uniformly bounded}) $\fg$ module is a weight module $M$ such that the dimension of weight spaces is uniformly bounded, that is there is $N\in\bN$ with $\dim M_\lambda<N$ for all $\lambda\in\Supp(M)$. Clearly, if $M$ is simple, then $\Supp(M)\subseteq\lambda+\frac{1}{2}\bZ$ for some $\lambda\in\bC$.

Let $(\fg,\cA)$ be $(\fk,A), (\fk^+,A)$ or $(\fk^+,\cA^+)$. A \emph{jet $\fg$ module} associated to $\cA$ is a weight $\cA\fg$ module with a finite dimensional weight space. Denote the category consisting of all jet $\fg$ modules associated to $\cA$ by $\sJ(\fg,\cA)$. Clearly, any simple module in $\sJ(\fk,A)$ and $\sJ(\fk^+,A)$ is cuspidal.

Let $\sigma: L\to L'$ be any homomorphism of Lie superalgebras or associative superalgebras, and $M$ be any $L'$-module. Then $M$ become an $L$-module, denoted by $M^\sigma$, under $x\cdot v:=\sigma(x)v, \forall x\in L, v\in M$. Denote by $T$ the automorphism of $L$ defined by $T(x):=(-1)^{|x|}x, \forall x\in L$. For any $L$-module $M$, $\Pi(M)$ is the module defined by a parity-change of $M$.

A module $M$ over an associative superalgebra $B$ is called \emph{strictly simple} if it is a simple module over the associative algebra $B$ (forgetting the $\bZ_2$-gradation).

We need the following result on tensor modules over tensor superalgebras.
\begin{lemm}[{\cite[Lemma 2.1, 2.2]{LX}}]\label{tensor}
Let $B,B'$ be associative superalgebras, and $M,M'$ be $B,B'$ modules, respectively.
\begin{enumerate}
\item $M\otimes M'\cong\Pi(M)\otimes\Pi(M'^T)$ as $B\otimes B'$-modules.
\item If in addition that $B'$ has a countable basis and $M'$ is strictly simple, then
\begin{enumerate}
\item $M\otimes M'$ is a simple $B\otimes B'$-module if and only if $M$ is a simple $B$-module.
\item If $V$ is a simple $B\otimes B'$-module containing a strictly simple $B'=\bC\otimes B'$ module $M'$, then $V\cong M\otimes M'$ for some simple $B$-module $M$.
\end{enumerate}
\end{enumerate}
\end{lemm}

The subalgebra of $\fk$ spanned by $\{L_k\,|\, k\in\bZ\}$ is isomorphic to the Witt algebra $W$. Denote by $W^+$ the subalgebra of $W$ spanned by $\{L_k\,|\,k\in-1+\bZ_+\}$. The following results for $W$ modules and $W^+$ modules will be used.
\begin{lemm}\label{Omegaoper}
Let $\Omega_{k, s}^{(m)}=\sum\limits_{i=0}^m(-1)^i\binom{m}{i}L_{k-i}L_{s+i}$.
\begin{enumerate}
\item ({\cite[Corollary 3.7]{BF}})
For every $\ell\in\bN$ there exists $m\in\bN$ such that for all $k, s\in\bZ$, $\Omega_{k,s}^{(m)}$ annihilate every cuspidal $W$-module with a composition series of length $\ell$.
\item (\cite[Corollary 4.4]{XL}). For every $r\in\bN$ there is an $m\in\bN$ such that for all $k,s\in\bZ_+$, $\Omega_{k+m-1,s-1}^{(m)}$ annihilates every cuspidal $W^+$ module $V$ with $\dim V_\lambda\leq r$ for all $\lambda\in\Supp(V)$.
\end{enumerate}
\end{lemm}

\section{Jet modules}
\label{cuspidalAK}

In this section, we will classify all simple jet modules in $\sJ(\fk, A), \sJ(\fk^+,A)$ and $\sJ(\fk^+,A^+)$. First, we have
\begin{lemm}
Let $M\in\sJ(\fk,A)$ be simple. Then $M$ is also a simple module in $\sJ(\fk^+,A)$, and the $A\fk$ module structure is uniquely determined by the $A\fk^+$ module structure.
\end{lemm}
\begin{proof}
Following from Lemma \ref{Omegaoper}, on $M$ we have
\begin{align*}
0=&\sum\limits_{i=0}^1\sum\limits_{j=0}^2(-1)^{i+j}\binom{2}{j}[t^{r+i-j},\Omega^{(m)}_{k+1-i,s-1+j}]=\sum\limits_{i=0}^{m+2}(-1)^i\binom{m+2}{i}t^{r+k+1-i}\cdot L_{s-1+i},
\end{align*}
where $r,k,s\in\bZ$. So $A\cdot\sum\limits_{i=0}^{n}(-1)^i\binom{n}{i}t^{k-i}\cdot L_{s+i}M=0$ for all $k,s\in\bZ, n\geq m+2$. Hence, on $M$ we have
\begin{align*}
0=&[\sum\limits_{i=0}^{m+3}(-1)^i\binom{m+3}{i}t^{k-i}\cdot L_{s-1+i},G_{p+1}]-[\sum\limits_{i=0}^{m+3}(-1)^i\binom{m+3}{i}t^{k-i}\cdot L_{s+i},G_p]\\
=&\sum\limits_{i=0}^{m+3}(-1)^i\binom{m+3}{i}(p+\frac{3-i-s}{2})t^{k-i}\cdot G_{p+s+i}-\sum\limits_{i=0}^{m+3}\binom{m+3}{i}(k-i)t^{p+k-i+\frac{1}{2}}\xi\cdot L_{s-1+i}\\
&-\Big(\sum\limits_{i=0}^{m+3}(-1)^i\binom{m+3}{i}(p-\frac{i+s}{2})t^{k-i}\cdot G_{p+s+i}-\sum\limits_{i=0}^{m+3}\binom{m+3}{i}(k-i)t^{p+k-i-\frac{1}{2}}\xi\cdot L_{s+i}\Big)\\
=&\frac{3}{2}\sum\limits_{i=0}^{m+3}(-1)^i\binom{m+3}{i}t^{k-i}\cdot G_{p+s+i},
\end{align*}
where $k,s\in\bZ, p\in\bZ+\frac{1}{2}$. So $A\cdot\sum\limits_{i=0}^{n}(-1)^i\binom{n}{i}t^{k-i}\cdot G_{p-i}M=0$ for all $n\geq m+3, k\in\bZ, p\in\bZ+\frac{1}{2}$. Therefore, the actions of $L_i (i\in\bZ),G_p(p\in\bZ+\frac{1}{2})$ are determined by the actions of $L_i(-1\leq i\leq m), G_p(-\frac{1}{2}\leq p\leq m+\frac{1}{2})$. And hence lemma follows.
\end{proof}

Thus, to determine all simple modules in $\sJ(\fk,A)$, it suffices to determine all simple modules in $\sJ(\fk^+,A)$.

For $n\in\bZ_+$, let
\begin{align*}
L_n'&:=\sum\limits_{i=0}^{n+1}(-1)^{i+1}\binom{n+1}{i}t^{n-i+1}\cdot L_{i-1}+\frac{n+1}{2}\sum\limits_{i=0}^n(-1)^i\binom{n}{i}t^{n-i}\xi\cdot G_{i-\frac{1}{2}},\\
G'_{n-\frac{1}{2}}&:=\sum\limits_{i=0}^n(-1)^i\binom{n}{i}(t^{n-i}\cdot G_{i-\frac{1}{2}}-2t^{n-i}\xi\cdot L_{i-1}).
\end{align*}
Then we have
\begin{align}
&\sum\limits_{k=0}^n(-1)^k\binom{n+1}{k+1}t^{n-k}\cdot(L'_k-\frac{k+1}{2}\xi\cdot G'_{k-\frac{1}{2}})+t^{n+1}\cdot L_{-1}\nonumber\\
=&\sum\limits_{k=0}^n\sum\limits_{i=0}^{k+1}(-1)^{k+i+1}\binom{n+1}{k+1}\binom{k+1}{i}t^{n-i+1}\cdot L_{i-1}+t^{n+1}\cdot L_{-1}\nonumber\\
=&\sum\limits_{j=0}^n\sum\limits_{k=j}^n(-1)^{k+j}\binom{n+1}{k+1}\binom{k+1}{j+1}t^{n-j}\cdot L_j+\sum\limits_{j=0}^{n+1}(-1)^j\binom{n+1}{j}t^{n+1}\cdot L_{-1}\nonumber\\
=&\sum\limits_{j=0}^n(-1)^j\binom{n+1}{j+1}\sum\limits_{k=j}^n(-1)^k\binom{n-j}{k-j}t^{n-j}\cdot L_j\nonumber\\
=&L_n,\label{L}
\end{align}
and
\begin{align}
&\sum\limits_{k=0}^n(-1)^k\binom{n}{k}t^{n-k}\cdot(G'_{k-\frac{1}{2}}-2\xi\cdot L'_{k-1})=\sum\limits_{k=0}^n\sum\limits_{i=0}^k(-1)^{k+i}\binom{n}{k}\binom{k}{i}t^{n-i}\cdot G_{i-\frac{1}{2}}\nonumber\\
=&\sum\limits_{i=0}^n\sum\limits_{k=i}^n(-1)^{k+i}\binom{n}{i}\binom{n-i}{k-i}t^{n-i}\cdot G_{i-\frac{1}{2}}=G_{n-\frac{1}{2}}.\label{G}
\end{align}

Let $\cT$ be the supersubspace of $\overline{U}^+$ with a basis $\fB=\{L'_n,G'_{m-\frac{1}{2}}\,|\,n\in\bZ_+, m\in\bN\}$. Clearly, $\cT$ is also a supersubspace of $\overline{U}^{++}$. We have
\begin{lemm}\label{basis}
\begin{enumerate}
\item $\widetilde{\fB}=\fB\cup\{G_{-\frac{1}{2}},L_{-1}\}$ is a basis of the free left $A$ ($A^+$, respectively) module $A\cdot \fk^+$ ($A^+\cdot \fk^+$, respectively).
\item $\cT=\{x\in A\cdot \fk^+\,|\,[x,A]=[x,G_{-\frac{1}{2}}]=0\}=\{x\in A^+\cdot \fk^+\,|\,[x,A^+]=[x,G_{-\frac{1}{2}}]=0\}.$
Thus, $\cT$ is a Lie supersubalgebra of $\overline{U}^+$ as well as a Lie supersubalgebra of $\overline{U}^{++}$.
\end{enumerate}
\end{lemm}
\begin{proof}
\begin{enumerate}
\item The $A$-linear independence of $\widetilde{\fB}$ is easy to check. And the first statement follows from \eqref{L} and \eqref{G}.
\item For the second statement, first we have
\begin{align*}
&[G_{-\frac{1}{2}},L'_n]\\
=&[G_{-\frac{1}{2}},\sum\limits_{i=0}^{n+1}(-1)^{i+1}\binom{n+1}{i}t^{n-i+1}\cdot L_{i-1}+\frac{n+1}{2}\sum\limits_{i=0}^n(-1)^i\binom{n}{i}t^{n-i}\xi\cdot G_{i-\frac{1}{2}}]\\
=&\sum\limits_{i=0}^{n+1}(-1)^{i+1}\binom{n+1}{i}((n-i+1)t^{n-i}\xi\cdot L_{i-1}+\frac{i}{2}t^{n-i+1}\cdot G_{i-\frac{3}{2}})\\
&+\frac{n+1}{2}\sum\limits_{i=0}^n(-1)^i\binom{n}{i}(-t^{n-i}\cdot G_{i-\frac{1}{2}}+2t^{n-i}\xi\cdot L_{i-1})=0,\\
&[L_n',t^k]\\
=&\sum\limits_{i=0}^{n+1}(-1)^{i+1}\binom{n+1}{i}t^{n-i+1}\cdot [L_{i-1},t^k]+\frac{n+1}{2}\sum\limits_{i=0}^n(-1)^i\binom{n}{i}t^{n-i}\xi\cdot [G_{i-\frac{1}{2}},t^k]\\
=&\sum\limits_{i=0}^{n+1}(-1)^{i+1}\binom{n+1}{i}t^{n-i+1}\cdot kt^{k+i-1}+\frac{n+1}{2}\sum\limits_{i=0}^n(-1)^i\binom{n}{i}t^{n-i}\xi\cdot kt^{i+k-1}\xi=0,\\
&[L_n',t^k\xi]\\
=&\sum\limits_{i=0}^{n+1}(-1)^{i+1}\binom{n+1}{i}t^{n-i+1}\cdot [L_{i-1},t^k\xi]+\frac{n+1}{2}\sum\limits_{i=0}^n(-1)^i\binom{n}{i}t^{n-i}\xi\cdot [G_{i-\frac{1}{2}},t^k\xi]\\
=&\sum\limits_{i=0}^{n+1}(-1)^{i+1}\binom{n+1}{i}t^{n-i+1}\cdot (kt^{k+i-1}\xi+\frac{i}{2}t^{k+i}\xi)-\frac{n+1}{2}\sum\limits_{i=0}^n(-1)^i\binom{n}{i}t^{n-i}\xi\cdot t^{i+k}\\
=&0,\\
&[G_{-\frac{1}{2}},G'_{n-\frac{1}{2}}]\\
=&\sum\limits_{i=0}^n(-1)^i\binom{n}{i}[G_{-\frac{1}{2}},t^{n-i}\cdot G_{i-\frac{1}{2}}-2t^{n-i}\xi\cdot L_{i-1}]\\
=&\sum\limits_{i=0}^{n}(-1)^i\binom{n}{i}((n-i)t^{n-i-1}\xi\cdot G_{i-\frac{1}{2}}-2t^{n-i}\cdot L_{i-1}+2t^{n-i}\cdot L_{i-1}+it^{n-i}\xi\cdot G_{i-\frac{3}{2}})\\
=&\sum\limits_{j=1}^n(-1)^{j+1}\binom{n-1}{j-1}nt^{n-j}\xi\cdot G_{j-\frac{3}{2}}+\sum\limits_{i=1}^n(-1)^in\binom{n-1}{i-1}t^{n-i}\xi\cdot G_{i-\frac{3}{2}}=0,\\
&[G'_{n-\frac{1}{2}},t^k]\\
=&\sum\limits_{i=0}^n(-1)^i\binom{n}{i}(t^{n-i}\cdot[G_{i-\frac{1}{2}},t^k]-2t^{n-i}\xi\cdot[L_{i-1},t^k])\\
=&-k\sum\limits_{i=0}^n(-1)^i\binom{n}{i}t^{n+k-1}\xi=0,\\
&[G'_{n-\frac{1}{2}},t^k\xi]\\
=&\sum\limits_{i=0}^n(-1)^i\binom{n}{i}(t^{n-i}\cdot[G_{i-\frac{1}{2}},t^k\xi]-2t^{n-i}\xi\cdot[L_{i-1},t^k\xi])\\
=&-\sum\limits_{i=0}^n(-1)^i\binom{n}{i}t^{n+k}=0.
\end{align*}
That is $\cT\subseteq\cT_1=\{x\in A\cdot \fk^+\,|\,[x,A]=[x,G_{-\frac{1}{2}}]=0\}$. On the other hand, let $x=\sum f_i\cdot x_i+x'\in\cT_1$ with $f_i\in A, x_i\in\fB,x'\in A\cdot\{G_{-\frac{1}{2}},L_{-1}\}$. Then
\[
0=[G_{-\frac{1}{2}},x]=\sum[G_{-\frac{1}{2}},f_i]\cdot x_i+[G_{-\frac{1}{2}},x']
\]
implies that $f_i\in\bC$ and $x'=a(G_{-\frac{1}{2}}-2\xi\cdot L_{-1})+bL_{-1}$ for some $a,b\in\bC$. And $[x,f]=[x',f]=0$ for all $f\in A$ implies that $a=b=0$, that is $x'=0$. So $\cT_1\subseteq\cT$ and hence $\cT=\cT_1$. \qedhere
\end{enumerate}
\end{proof}

\begin{lemm}
We have the associative superalgebra isomorphism
\begin{align*}
\iota: &\cK\otimes U(\cT)\to\overline{U}^+,\,\iota(x\otimes y)=x\cdot y,
\end{align*}
where $\cK=\bC[t^{\pm1},\xi,\ptl_t,\ptl_\xi]$ is the Weyl superalgebra, and $x\in\cK, y\in U(\cT)$. Moreover, the restriction of $\iota$ to $\cK^+\otimes U(\cT)$ is an isomorphism from $\cK^+\otimes U(\cT)$ to $\overline{U}^{++}$, where $\cK^+=A^+[\ptl_t,\ptl_\xi]$.
\end{lemm}
\begin{proof}
Consider the map $\iota': A[G_{-\frac{1}{2}}]\otimes U(\cT)\to\overline{U}^+,\,\iota(x\otimes y)=x\cdot y$. Since $\cT$  is a Lie supersubalgebra of $\overline{U}^+$ and $A[G_{-\frac{1}{2}}]$ is an associative supersubalgebra of $\overline{U}^+$, the restrictions of $\iota'$ on $A[G_{-\frac{1}{2}}]$ and $U(\cT)$ are well-defined. From Lemma \ref{basis}, in $\overline{U}$, $\iota'(U(\cT))$ and $\iota'(A[G_{-\frac{1}{2}}])$ are super commutative, hence $\iota'$ is a well-defined homomorphism of associative superalgebras. Let $\fL=A\otimes\cT+(A[G_{-\frac{1}{2}}]+A)\otimes\bC\subseteq A[G_{-\frac{1}{2}}]\otimes U(\cT)$. Equations \eqref{L} and \eqref{G} tell us that $\iota'|_{\fL}:\fL\to A\cdot \fk^++A$ is a Lie superalgebra isomorphism. Thus, we have a Lie superalgebra homomorphism $\eta:\widetilde{\fk}^+\xrightarrow{(\iota'|_\fL)^{-1}}\fL\subseteq A[G_{-\frac{1}{2}}]\otimes U(\cT)$ with
\begin{align*}
\eta(t^n\xi^r)=&t^n\xi^r\otimes1, n\in\bZ, r=0,1;\\
\eta(L_n)=&\sum\limits_{k=1}^n(-1)^k\binom{n+1}{k+1}(t^{n-k}\otimes L'_k-\frac{k+1}{2}t^{n-k}\xi\otimes G'_{k-\frac{1}{2}})+(n+1)t^n\otimes L'_0\\
&-\frac{n+1}{2}t^n\xi\cdot G_{-\frac{1}{2}}\otimes1+t^{n+1}\cdot L_{-1}\otimes1, n\geq-1;\\
\eta(G_{n-\frac{1}{2}})=&\sum\limits_{k=1}^n(-1)^k\binom{n}{k}(t^{n-k}\otimes G'_{k-\frac{1}{2}}-2t^{n-k}\xi\otimes L'_{k-1})+t^n\cdot G_{-\frac{1}{2}}\otimes1-2t^n\xi\cdot L_{-1}\otimes1.
\end{align*}
So we have the associative superalgebra homomorphism $\tilde{\eta}:U^+\to A[G_{-\frac{1}{2}}]\otimes U(\cT)$ with $\tilde{\eta}|_{\widetilde{K}^+}=\eta$ and $I^+\subseteq\Ker\tilde{\eta}$. Therefore, we have the induced associative superalgebra homomorphism $\bar{\eta}: \overline{U}^+\to A[G_{-\frac{1}{2}}]\otimes U(\cT)$ with $\bar{\eta}=\iota'^{-1}$. Hence $\iota'$ is an isomorphism and lemma follows from the fact that $A[G_{-\frac{1}{2}}]$ is isomorphic to $\cK$.
\end{proof}

\begin{lemm}
We have the Lie superalgebra isomorphism $\cT\cong \fk_{\ge0}$.
\end{lemm}
\begin{proof}
Consider the linear map $\psi: \cT\to \fk_{\ge0}$ defined by $\psi(L'_n)\to(-1)^nLn,\psi(G'_{n+\frac{1}{2}})=(-1)^{n+1}G_{n+\frac{1}{2}}$. Then clearly $\psi$ is an isomorphism of vector superspaces. From \eqref{L} and \eqref{G}, we have $\fk_{\ge0}\subseteq\fm\cdot L_{-1}+\fm_1\cdot G_{-\frac{1}{2}}+A^+\cdot\cT$, where $\fm=A^+t$ and $\fm_1=\fm+\bC\xi$. Hence, we have the Lie superalgebra homomorphism
$\omega:\fm \fk^+\subseteq\fm\cdot L_{-1}+\fm_1\cdot G_{-\frac{1}{2}}+A^+\cdot\cT\to(\fm\cdot L_{-1}+\fm_1\cdot G_{-\frac{1}{2}}+A^+\cdot\cT)/(\fm\cdot L_{-1}+\fm_1\cdot G_{-\frac{1}{2}}+\fm_1\cdot\cT)\to(A^+\cdot\cT)/(\fm_1\cdot\cT)\to\cT
$. More precisely, we have $\omega(L_n)=(-1)^nL'_n,\omega(G_{n+\frac{1}{2}})=(-1)^{n+1}G_{n+\frac{1}{2}}$. So $\psi=\omega^{-1}$ and $\omega$ is a Lie superalgebra isomorphism, so is $\psi$.
\end{proof}

The following results on simple weight $\cK$ modules and simple weight $\cK^+=\bC[t,\xi,\ptl_t,\ptl_\xi]$ modules, on which $t\ptl_t$ acts diagonalizable, is needed.
\begin{lemm}\label{weyl}
\begin{enumerate}
\item Up to a parity change, any simple weight $\cK$ module is isomorphic to some strictly simple $\cK$ module $A(\lambda):=t^\lambda\bC[t^{\pm1},\xi]$ for some $\lambda\in\bC$.
\item Up to a parity change, any simple weight $\cK^+$ module is isomorphic to one of the following strictly simple $\cK^+$ module:
    \[
    t^\lambda\bC[t^{\pm1},\xi], \bC[t,\xi], \bC[t^{\pm1},\xi]/\bC[t],
    \]
    where $\lambda\in\bC\setminus\bZ$.
\end{enumerate}
\end{lemm}
\begin{proof}
The first statement is just Lemma 3.5 in \cite{LX}. For the second statement, it is easy to check the giving modules are strictly simple $\cK^+$ modules. Now let $V$ be a simple weight $\cK^+$ module with $\lambda\in\Supp(V)$. For a fixed nonzero homogeneous element $v\in V_\lambda$, since $V'=\bC[\ptl_\xi]v$ is finite dimensional with $\ptl_\xi$ acting nilpotently, we can find a nonzero homogeneous element $v'\in V'$ such that $\ptl_\xi v'=0$. Then $\bC[\xi,\ptl_\xi]v'=\bC[\xi]v'$ is a strictly simple $\bC[\xi,\ptl_\xi]$ module. So lemma follows from Lemma \ref{tensor} and the facts that $\cK^+\cong\bC[t,\ptl_t]\otimes\bC[\xi,\ptl_\xi]$ and  any simple weight $\bC[t,\ptl_t]$ module is isomorphic one of the following (cf. \cite{FGM}): $t^\lambda\bC[t^{\pm1}],\bC[t],\bC[t^{\pm1}]/\bC[t]$.
\end{proof}

Now for any $\fk_{\ge0}$ module, we have the $A\fk^+$ module $\Gamma(\lambda,V):=(A(\lambda)\otimes V)^\varphi$, where $\varphi: \overline{U}^+\xrightarrow{\iota^{-1}}\cK\otimes U(\cT)\xrightarrow{\mathrm{id}\otimes\psi}\cK\otimes U(\fk_{\ge0})$ and the $A^+\fk^+$ modules $\cF(P,V)=(P\otimes V)^{\varphi|_{\overline{U}^{++}}}$ with $P$ a simple weight $\cK^+$ module. More precisely, $\Gamma(\lambda,V)=A\otimes V$ ($\cF(P,V)=P\otimes V$, respectively) with actions
\begin{align*}
x\cdot(y\otimes v)=&xy\otimes v,\\
L_n\cdot(y\otimes v)=&\sum\limits_{k=1}^{n}\binom{n+1}{k+1}(t^{n-k}y\otimes L_k\cdot v-\frac{k+1}{2}t^{n-k}\xi y\otimes G_{k-\frac{1}{2}}\cdot v)\\
&+(n+1)t^ny\otimes L_0\cdot v+\frac{n+1}{2}t^n\xi\ptl_\xi(y)\otimes v+t^{n}(t\ptl_t(y)+\lambda y)\otimes v,\\
G_{n+\frac{1}{2}}\cdot(y\otimes v)=&\sum\limits_{k=1}^{n+1}\binom{n+1}{k}(t^{n+1-k}y\otimes G_{k-\frac{1}{2}}\cdot v-2t^{n+1-k}\xi y\otimes L_{k-1}\cdot v)\\
&-t^{n}\xi (t\ptl_t(y)+\lambda y)\otimes v-t^{n+1}\ptl_\xi(y)\otimes v,
\end{align*}
where $x\in A$ ($A^+$, respectively), $y\in A$ ($P$, respectively), $n\in\bZ_{\geq-1}$ and $v\in V$.

\begin{lemm}
\begin{enumerate}
\item For any $\lambda\in\bC$ and any simple $\fk_{\geq0}$ module $V$, the $A\fk^+$ module $\Gamma(\lambda,V)$ is simple.
\item For any simple weight $\cK^+$ module $P$ and any simple $\fk_{\geq0}$ module $V$, the $A^+\fk^+$ module $\cF(P,V)$ is simple.
\item Any simple module in $\sJ(\fk^+,A)$ is isomorphic to some $\Gamma(\lambda, V)$ for some $\lambda\in\bC$ and some simple weight $\fk_{\ge0}$ module $V$.
\item  Any simple module in $\sJ(\fk^+,A^+)$ is isomorphic to some $\cF(P,V)$ for some simple weight $\cK^+$ module $P$ and some simple weight $\fk_{\ge0}$ module $V$.
\end{enumerate}
\end{lemm}
\begin{proof}
Statement 1 and 2 follow from Lemma \ref{weyl} and Lemma \ref{tensor}. Now let $M\in\sJ(\fk^+,A)$ ($\sJ(\fk^+,A^+)$, respectively) be simple with $\dim M_\lambda<\infty$  for some $\lambda\in\Supp(M)$. Then $M^{\varphi^{-1}}$ is a simple $\cK\otimes U(\fk_{\ge0})$ ($\cK^+\otimes U(\fk_{\ge0})$, respectively) module. If $M\in\sJ(\fk^+,A)$, we can find a common eigenvector $v$ for $t\ptl_t$ and $L_0$ such that $(t\ptl_t-\lambda)v=\ptl_\xi v=0$, and hence $\cK v\cong A(\lambda)$ or $\cK v\cong \Pi(A(\lambda))$. From Lemma \ref{tensor}, there exists a simple $U(\fk_{\ge0})$ module $V$ such that $M^{\varphi^{-1}}\cong A(\lambda)\otimes V$ or $M^{\varphi^{-1}}\cong\Pi(A(\lambda))\otimes V\cong A(\lambda)\otimes\Pi(V)$. The fact that the adjoint action of $L_0$ on $\cK\otimes U(\fk_{\ge0})$ is diagonalizable tells us that $V$ and $\Pi(V)$ is a simple weight $\fk_{\ge0}$ module.

For Statement 4, we can find a nonzero element $u$ such that $(t\ptl_t-\lambda)\cdot u=\ptl_\xi\cdot u=0 (\lambda\notin\bZ)$ or $\ptl_t\cdot u=\ptl_\xi\cdot u=0$ or $t\cdot u=\ptl_\xi\cdot u=0$, and hence $\cK^+u\cong t^\lambda\bC[t^{\pm1},\xi]$ or $\cK^+u\cong\bC[t,\xi]$ or $\cK^+u\cong\bC[t^{\pm1},\xi]/\bC[t]$. And again by Lemma \ref{tensor}, $M^{\varphi^{-1}}=(\cK^+\otimes U(\fk_{\ge0}))u\cong \cK^+u\otimes V$ for some simple $U(\fk_{\ge0})$ module $V$. Moreover, since the adjoint actions of $t\ptl_t$ and $L_0$ on $\cK^+\otimes U(\fk_{\ge0})$ is diagonalizable, $L_0$ is diagonalizable on $M^{\varphi^{-1}}$, and hence $V$ is a simple weight $\fk_{\ge0}$ module.
\end{proof}

\begin{lemm}
Suppose $V$ is a simple weight $\fk_{\ge0}$ module. Then $\fk_{\ge\frac{1}{2}}V=0$. Thus, $V$ can be regarded as a simple weight $\fk_{\geq0}/\fk_{\geq\frac{1}{2}}$ module. Moreover, $V=\bC v$ is one dimensional with $L_0v=bv$ for some $b\in\bC$,  and $L_iv=G_{i-\frac{1}{2}}v=0, \forall i\geq1$.
\end{lemm}
\begin{proof}
Let $V$ be a simple weight $\fk_{\ge0}$ module and $0\neq v\in V_\lambda$ for some $\lambda\in\Supp(V)$. Then $V=U(\fk_{\ge0})v$ and $\Supp(V)\subseteq\lambda+\frac{1}{2}\bZ_+$. On the other hand, $\fk_{\ge\frac{1}{2}}V$ is a submodule of $V$ with support $\lambda+\frac{1}{2}\bN$ so that $v\notin \fk_{\ge\frac{1}{2}}v$. Therefore, from the simplicity of $V$, we have $\fk_{\ge\frac{1}{2}}V=0$.
\end{proof}

\begin{coro}
Let $P$ be a simple weight $\cK$ ($\cK^+$, respectively) module and $V$ be a simple weight $\fk_{\ge0}/\fk_{\ge\frac{1}{2}}$ module. Then $\cF(P,V)$ is a cuspidal $A\fk^+$ ($A^+\fk^+$, respectively) module.
\end{coro}
\begin{rema}
Clearly, any $\Gamma(\lambda,V)$ is an $A^+\fk^+$ module. It is easy to check as $\Gamma(\lambda,V)$ is simple as $A^+\fk^+$ module if and only if $\lambda\notin\bZ$. And any simple module in $\sJ(\fk^+,A^+)$ is a quotient of some $\Gamma(\lambda,V)$. In particular, any simple module in $\sJ(\fk^+, A^+)$ is cuspidal.
\end{rema}

Now we can classify all simple modules in $\sJ(\fk,A), \sJ(\fk^+,A)$ and $\sJ(\fk^+,A^+)$. 
\begin{theo}
\begin{enumerate}
\item Any simple module in $\sJ(\fk,A)$ as well as in $\sJ(\fk^+,A)$ is isomorphic to the module $\Gamma(\lambda,b)=A$ for some $\lambda,b\in\bC$ with actions: for all $n,k\in\bZ$,
    \begin{align}
    &x\cdot y=xy,\, \forall x, y\in A,\label{actionA}\\
    &L_n\cdot t^k=(\lambda+k+b(n+1))t^{n+k},\\
    &L_n\cdot t^k\xi=(\lambda+k+(n+1)(b+\frac{1}{2}))t^{n+k}\xi,\\
    &G_{n+\frac{1}{2}}\cdot t^k=-(k+\lambda+2b(n+1))t^{n+k}\xi,\\
    &G_{n+\frac{1}{2}}\cdot t^k\xi=-t^{n+k+1}.\label{acionG}
    \end{align}
\item Any simple module in $\sJ(\fk^+,A^+)$ is isomorphic to one of the following
    \begin{enumerate}
    \item $\Gamma(\lambda,b)$ for some $\lambda\in\bC\setminus\bZ$ and $b\in\bC$;
    \item The $A^+\fk^+$ submodule $\Gamma^+(0,b):=A^+$ of $\Gamma(0,b)$ for some $b\in\bC$;
    \item The $A^+\fk^+$ quotient module $\Gamma^-(0,b):=\Gamma(0,b)/\Gamma^+(0,b)$.
    \end{enumerate}
\end{enumerate}
\end{theo}


\section{Main results}
\label{main}

In this section, we will classify all simple weight $\widehat{\fk}$ modules and $\fk^+$ modules with finite dimensional weight spaces. First of all, from the representation theory of Virasoro algebra, we know that $C$ acts trivially on any simple cuspidal $\widehat{\fk}$ module, and hence the category of simple cuspidal $\widehat{\fk}$ modules is naturally equivalent to the category of simple cuspidal $\fk$ modules. 

The following result is well-known
\begin{lemm}\label{weightupper}
Let $M$ be a weight module with finite dimensional weight spaces for the Virasoro algebra with $\mathrm{supp}(M)\subseteq\lambda+\bZ$. If for any $v\in M$, there exists $N(v)\in\bN$ such that $L_iv=0, \forall i\geq N(v)$, then $\mathrm{supp}(M)$ is upper bounded.
\end{lemm}

Lemma \ref{appN=1} tells us that classification of simple cuspidal modules is an important step in classifying all simple Harish-Chandra modules over $\widehat{\fk}$ and $\fk^+$.
\begin{lemm}\label{appN=1}
\begin{enumerate}
\item Any simple weight $\widehat{\fk}$ module with finite dimensional weight spaces which is not cuspidal is a highest/lowest weight module.
\item Any simple weight $\fk^+$ module with finite dimensional weight spaces is cuspidal.
\end{enumerate}
\end{lemm}
\begin{proof}
\begin{enumerate}
\item For a simple weight $\widehat{\fk}$ module $M$, fix a $\lambda\in\Supp(M)$. Suppose $M$ is not cuspidal, then there exists a $k\in\frac{1}{2}\bZ$ such that $\dim M_{-k+\lambda}>2(\dim M_\lambda+\dim M_{\lambda+\frac{1}{2}}+\dim M_{\lambda+1})$. Without lost of generality, we may assume that $k\in\bN$. Then there exists a nonzero element $w\in M_{-k+\lambda}$ such that $L_kw=L_{k+1}w=G_{k+\frac{1}{2}}w=0$. Therefore $L_iw=G_{i-\frac{1}{2}}w=0, \forall i\geq (k+1)^2$.

    Since $M'=\{v\in M\,|\, L_iv=G_{i-\frac{1}{2}}v=0, \forall i\gg0\}$ is a nonzero submodule of $M$, we know $M=M'$. And hence by Lemma \ref{weightupper}, $\Supp(M)$ is upper bounded, that is $M$ is a highest weight module.
\item Let $M$ be a simple weight $\fk^+$ module with finite dimensional weight spaces. First we show that if $M$ is not cuspidal, then $M$ is a highest weight module or a lowest weight module. If the action of $G_{-\frac{1}{2}}$ is not injective, then there exists $v\in M$ such that $G_{-\frac{1}{2}}v=0$, and clearly $M$ is a lowest weight module. So we may assume that $G_{-\frac{1}{2}}$ acts injectively on $M$, then we have $\dim M_{i-\frac{1}{2}}\geq\dim M_i$. Consider the restricted dual $M^*:=\bigoplus\limits_{\lambda\in\Supp(M)}\mathrm{Hom}_{\bC}(M_\lambda,\bC)$. Then $M^*$ is a $\fk^+$ module with support $-\Supp(M)$ under
    \[
    (x\cdot f)(v):=-(-1)^{|x||f|}f(xv), \forall x\in (K^+)_i, v\in M_{\lambda+j}, f\in\mathrm{Hom}_\bC(M_{\lambda+i+j},\bC).
    \]
    More precisely, $(M^*)_{-\lambda}=\mathrm{Hom}_{\bC}(M_\lambda,\bC)$. Hence, $\dim(M^*)_{-\lambda}=\dim M_{\lambda}$. Since $M$ is not cuspidal, there exists $\mu$ such that $\dim M_{\mu-\frac{1}{2}}>\dim M_\mu$ and hence the action of $G_{-\frac{1}{2}}$ on $M^*$ is not injective. Therefore, $M^*$ is a lowest weight module, that is $-\Supp(M)$ is lower bounded. So $\Supp(M)$ is upper bounded.

    To complete the proof, it remains to show any simple highest weight $\fk^+$ module is cuspidal. This is clear since from the PBW theorem, as vector spaces, $M\cong\bC[G_{-\frac{1}{2}}]v$, which is cuspidal with $\dim M_\lambda\leq1$ for all $\lambda\in\Supp(M)$.\qedhere
\end{enumerate}
\end{proof}


Now let $(\fg,\cA)$ be $(\fk,A)$ or $(\fk^+,A^+)$ and $M$ be a cuspidal $\fg$ module.
Consider $\fg$ as the adjoint $\fg$ module. Then the tensor product $\fg\otimes M$  becomes an $\cA\fg$ module under
\[
x\cdot (y\otimes u)=(xy)\otimes u, \forall x\in\cA,y\in\fg, u\in M.
\]
Denote $\fK(M)=\{\sum\limits_{i}x_i\otimes v_i\in \fg\otimes M\,|\,\sum\limits_{i}(ax_i)v_i=0, \forall a\in\cA\}$. Then it is easy to see that $\fK(M)$ is an $\cA\fg$ submodule of $\fg\otimes M$. And hence we have the $\cA\fg$ module $\widehat{M}=(\fg\otimes M)/\fK(M)$. Also, we have a $\fg$ module epimorphism defined by
\[\pi: \widehat{M}\to \fg M; x\otimes y+\fK(M)\mapsto xy, \forall x\in \fg, y\in M.\]
$\widehat{M}$ is called the $A$-cover of $M$ if $\fg M=M$.

Let $M$ be a cuspidal $\fk$ module. Then $M$ is a cuspidal $W$ module and hence there exists $m\in\bN$ such that for all $k,p\in\bZ$, $\Omega_{k,p}^{(m)}M=0$. Therefore $[\Omega_{k,p}^{m},G_j]M=0, \forall k,p\in\bZ, j\in\frac{1}{2}+\bZ$. Then, on $M$ we have
\begin{align*}
0=&[\Omega_{k,p-1}^{(m)},G_{j+1}]-2[\Omega_{k,p}^{(m)},G_j]+[\Omega_{k,p+1}^{(m)},G_{j-1}]-[\Omega_{k+1,p-1}^{(m)},G_{j}]\\
&+2[\Omega_{k+1,p}^{(m)},G_{j-1}]-[\Omega_{k+1,p+1}^{(m)},G_{j-2}]\\
=&[\sum\limits_{i=0}^m(-1)^i\binom{m}{i}L_{k-i}L_{p-1+i},G_{j+1}]-2[\sum\limits_{i=0}^m(-1)^i\binom{m}{i}L_{k-i}L_{p+i},G_j]\\
&+[\sum\limits_{i=0}^m(-1)^i\binom{m}{i}L_{k-i}L_{p+1+i},G_{j-1}]-[\sum\limits_{i=0}^m(-1)^i\binom{m}{i}L_{k+1-i}L_{p-1+i},G_{j}]\\
&+2[\sum\limits_{i=0}^m(-1)^i\binom{m}{i}L_{k+1-i}L_{p+i},G_{j-1}]-[\sum\limits_{i=0}^m(-1)^i\binom{m}{i}L_{k+1-i}L_{p+1+i},G_{j-2}]\\
=&\sum\limits_{i=0}^{m}(-1)^i\binom{m}{i}\Big((j+1-\frac{k-i}{2})G_{k-i+j+1}L_{p-1+i}+(j+1-\frac{p-1+i}{2})L_{k-i}G_{p+i+j}\\
&-2(j-\frac{k-i}{2})G_{k-i+j}L_{p+i}-2(j-\frac{p+i}{2})L_{k-i}G_{p+i+j}+(j-1-\frac{k-i}{2})G_{k-i+j-1}L_{p+i+1}\\
&+(j-1-\frac{p+i+1}{2})L_{k-i}G_{p+i+j}-(j-\frac{k-i+1}{2})G_{k-i+j+1}L_{p+i-1}\\
&-(j-\frac{p+i-1}{2})L_{k-i+1}G_{p+i+j-1}+2(j-1-\frac{k-i+1}{2})G_{k-i+j}L_{p+i}\\
&+2(j-1-\frac{p+i}{2})L_{k-i+1}G_{p+i+j-1}-(j-2-\frac{k-i+1}{2})G_{k-i+j-1}L_{p+i+1}\\
&-(j-2-\frac{p+i+1}{2})L_{k+1-i}G_{p+i+j-1}\Big)\\
=&\frac{3}{2}\sum\limits_{i=0}^{m}(-1)^i\binom{m}{i}(G_{k-i+j+1}L_{p+i-1}-2G_{k-i+j}L_{p+i}+G_{k-i+j-1}L_{p+i+1})\\
=&\frac{3}{2}\sum\limits_{i=0}^{m+2}(-1)^i\binom{m+2}{i}G_{k-i+j+1}L_{p+i-1}.
\end{align*}
Similarly, replacing $k$ by $k+m-1$ with $k,p\in\bZ_+$, we know that any cuspidal $\fk^+$ module is annihilated by $\sum\limits_{i=0}^m(-1)^i\binom{m+2}{i}G_{m+k-i+j}L_{p+i-1}$ for some $m\in\bN$ and for all $k,p\in\bZ_+$. That is we have
\begin{lemm}\label{omega}
\begin{enumerate}
\item Any cuspidal $\fk$ module is annihilated by $\sum\limits_{i=0}^m(-1)^i\binom{m}{i}G_{k-i}L_{p+i}$ for some $m\in\bN$ and for all $k\in\frac{1}{2}+\bZ,p\in\bZ$.
\item Any cuspidal $\fk^+$ module is annihilated by $\sum\limits_{i=0}^m(-1)^i\binom{m}{i}G_{m+k-i-1}L_{p+i-1}$ for some $m\in\bN$ and for all $k\in\frac{1}{2}+\bZ_+,p\in\bZ_+$.
\end{enumerate}
\end{lemm}

\begin{lemm}\label{covercuspidal}
Let $\fg$ be $\fk$ or $\fk^+$. For any simple cuspidal $\fg$ module $M$, $\widehat{M}$ is cuspidal.
\end{lemm}
\begin{proof}
It is obvious if $M$ it trivial. Now suppose $M$ is nontrivial. Then $\fg M=M$ and $\Supp(M)\subseteq \lambda+\frac{1}{2}\bZ$. Suppose $\dim M_\mu\leq r, \forall \mu\in\Supp(M)$. For $\fg=\fk$ ($\fk^+$, respectively), from Lemma \ref{Omegaoper} and Lemma \ref{omega}, there exists $m\in\bN$, such that such that $\sum\limits_{i=0}^m(-1)^i\binom{m}{i}L_{j-i}L_{p+i}v=\sum\limits_{i=0}^m(-1)^i\binom{m}{i}G_{j-i+\frac{1}{2}}L_{p+i}v=0, \forall j\in\bZ (\bZ_++m, \mbox{ respectively}),p\in\bZ (\bZ_+, \mbox{ respectively}), v\in M$. Hence,
\begin{equation}\label{induction}
\sum\limits_{i=0}^m(-1)^i\binom{m}{i}L_{j-i}\otimes L_{p+i}v,\sum\limits_{i=0}^m(-1)^i\binom{m}{i}G_{j-i+\frac{1}{2}}\otimes L_{p+i}v\in \fK(M).
\end{equation}

Let $S=\mathrm{span}\{L_{k},G_{k+\frac{1}{2}}\,|\,0\leq k\leq m\}$. Then $\dim S=2(m+1)$ and $S\otimes M$ is a $\bC L_0$ submodule of $\fg\otimes M$ with \[\dim(S\otimes M)_\mu\leq 2(m+1)r, \forall \mu\in\lambda+\frac{1}{2}\bZ.\]
We will prove that $\fg\otimes M=S\otimes V+\fK(M)$, from which we know $\widehat{M}$ is cuspidal. Indeed, we will prove by induction that for all $u\in M_\mu$,
\[
L_n\otimes u, G_{n+\frac{1}{2}}\otimes u\in S\otimes M.
\]
We only prove the claim for $n>m$, the proof for $n<0$ is similar. Since $L_0$ acts on $M_\mu$ with a nonzero scalar, we can write $u=L_0v$ for some $v\in M_\mu$. Then by \eqref{induction} and induction hypothesis, we have
\begin{align*}
L_{n}\otimes L_0v&=\sum\limits_{i=0}^m(-1)^i\binom{m}{i}L_{n-i}\otimes L_iv-\sum\limits_{i=1}^m(-1)^i\binom{m}{i}L_{n-i}\otimes L_iv\in S\otimes M+\fK(M),\\
G_{n+\frac{1}{2}}\otimes L_0v&=\sum\limits_{i=0}^m(-1)^i\binom{m}{i}G_{n-i+\frac{1}{2}}\otimes L_iv-\sum\limits_{i=1}^m(-1)^i\binom{m}{i}G_{n-i+\frac{1}{2}}\otimes L_iv\\
&\in S\otimes M+\fK(M).\qedhere
\end{align*}
\end{proof}
Now we can classify all simple cuspidal $\fk$ modules and all simple cuspidal $\fk^+$ modules.
\begin{theo}\label{cuspidal}
Let $(\fg,\cA)$ be $(\fk,A)$ or $(\fk^+,A^+)$. Then up to a parity change, any nontrivial simple cuspidal $\fg$ module is isomorphic to a simple quotient of a simple cuspidal $\cA\fg$ module. More precisely, we have, up to a parity change,
\begin{enumerate}
\item any nontrivial simple cuspidal $\fk$ module is isomorphic to a simple quotient of $\Gamma(\lambda,b)$ for some $\lambda,b\in\bC$;
\item any nontrivial simple cuspidal $\fk^+$ module is isomorphic to a simple quotient of $\Gamma(\lambda,b)$, $\Gamma^+(0,b)$ or $\Gamma^-(0,b)$ for some $\lambda,b\in\bC$.
\end{enumerate}
\end{theo}
\begin{proof}
Let $M$ be any nontrivial simple cuspidal $\fg$ module. Then $\fg M=M$ and there is an epimorphism $\pi: \widehat{M}\to M$. From Lemma \ref{covercuspidal}, $\widehat{M}$ is cuspidal. Hence $\widehat{M}$ has a composition series of $\cA\fg$ submodules:
\[
0=\widehat{M}^{(0)}\subset\widehat{M}^{(1)}\subset\cdots\subset\widehat{M}^{(s)}=\widehat{M}
\]
with $\widehat{M}^{(i)}/\widehat{M}^{(i-1)}$ being simple $\cA\fg$ modules. Let $k$ be the minimal integer such that $\pi(\widehat{M}^{(k)})\neq0$. Then we have $\pi(\widehat{M}^{(k)})=M, \widehat{M}^{(k-1)}=0$ since $M$ is simple. So we have an $\fg$-epimorphism from the simple $\cA\fg$ module $\widehat{M}^{(k)}/\widehat{M}^{(k-1)}$ to $M$.
\end{proof}

It is straightforward to check that as a $\widehat{\fk}$ module, $\Gamma(\lambda,b)$ has a unique nontrivial sub-quotient which we denote by $\Gamma'(\lambda,b)$. More precisely, we have
\begin{prop}\label{irr}
\begin{enumerate}
\item As a $\widehat{\fk}$ module, $\Gamma(\lambda,b)$ is simple if and only if $\lambda\notin\bZ$ or $\lambda\in\bZ$ and $b\neq0,\frac{1}{2}$.
\item As $\widehat{\fk}$ modules, $\Gamma(\lambda_1,b_1)\cong\Gamma(\lambda_2,b_2)$ if and only if $\lambda_1-\lambda_2\in\bZ, b_1=b_2$ or $\lambda_1\notin\bZ,\lambda_1-\lambda_2\in\bZ, b_1=\frac{1}{2},b_2=0$.
\item $\Gamma'(0,0)\cong\Pi(\Gamma'(0,\frac{1}{2}))$, where $\Gamma'(0,0)=\Gamma(0,0)/\bC 1, \Gamma'(0,\frac{1}{2})=\mathrm{span}\{t^i, t^k\xi\,|\,i,k\in\bZ, k\neq-1\}$.
\item For any $b\in\bC$, $\Gamma(\lambda,b)(\lambda\notin\bZ),\Gamma^+(0,b)$ and $\Gamma^-(0,b)$ are simple $\fk^+$ modules.
\end{enumerate}
\end{prop}

Combining Lemma \ref{appN=1}, Theorem \ref{cuspidal} and Proposition \ref{irr}, we can get the following result, which is a super version of the classical result for the Virasoro algebra and its subalgebra $W^+$ (cf. \cite{Ma}). The result of $\widehat{\fk}$ was also given in \cite{S2} by much complicated calculations.

\begin{theo}
\begin{enumerate}
\item Any simple $\widehat{\fk}$ module with finite dimensional weight spaces is a highest weight module, lowest weight module, or is isomorphic to  $\Gamma'(\lambda,b), \Pi(\Gamma'(\lambda,b))$ for some $\lambda,b\in\bC$ (which is called a module of the intermediate series).
\item Up to a parity change, any simple $\widehat{\fk}^+$ module with finite dimensional weight spaces is  isomorphic to $\Gamma(\lambda,b)$, $\Gamma^+(0,b)$, or $\Gamma^-(0,b)$ for some $\lambda\in\bC\setminus\bZ, b\in\bC$.
\end{enumerate}
\end{theo}

\noindent{\bf Acknowledgement:} Y. Cai is partially supported by NSF of China (Grant 11801390) and High-level Innovation and Entrepreneurship Talents Introduction Program of Jiangsu Province of China.  R. L\"{u} is partially supported by NSF of China (Grant 11471233, 11771122, 11971440).

\

Y. Cai: Department of Mathematics, Soochow University, Suzhou 215006, P. R. China. Email: yatsai@mail.ustc.edu.cn

R. L\"{u}: Department of Mathematics, Soochow University, Suzhou 215006, P. R. China. Email: rlu@suda.edu.cn.


\section*{References}



\begin{thebibliography}{00}

\bibitem{B} Y. Billig. {\it Jet modules.} Canad. J. Math. 59 (2007), no. 4, 721-729.

\bibitem{BF} Y. Billig, V. Futorny. {\it Classification of irreducible representations of Lie algebra of vector fields on a torus}, J. reine angew. Math.,  720(2016): 199-216.




\bibitem{CP} V. Chari and A. Pressley, {\it Unitary representations of the Virasoro algebra and a
conjecture of Kac}, Compositio Mathematica, 67(1988), 315-342

\bibitem{CLL} Y. Cai, D. Liu, R. L\"{u}, {\it Classification of simple Hrish-Chandra modules over the $N=1$ Ramond algebra}. arXiv: 2007.04483, to appear in J. Algebra.

\bibitem{DLM} P. Desrosiers, L. Lapointe,  P. Mathieu, { Superconformal field theory and Jack superpolynomials}, JHEP, { 09} (2012), 037.


\bibitem{Rao} S. E. Rao, {\it Partial classification of modules for Lie algebra of diffeomorphisms of $d$-dimensional torus}, J. Math. Phys. 45 (2004), no. 8, 3322-3333.

\bibitem{FGM} V. Futorny, D. Grantcharov, V. Mazorchuk, {\it Weight modules over infinite dimensional Weyl algebras}, Proc. Amer. Math. Soc., 142 (2014), no. 9, 3049-3057.

\bibitem{IK1}  K. Iohara, Y. Koga, {\it Representation theory of Neveu-Schwarz and Remond algebras I: Verma
modules}. Adv. Math. 177(2003), 61-69.

\bibitem{IK2} K. Iohara, Y. Koga,   {\it Representation theory of Neveu-Schwarz and Remond algebras II: Fock
modules}. Ann. Inst. Fourier 53 (2003), 1755-1818.




\bibitem {Ka} V. Kac, {\it Some problems of infinite-dimensional Lie algebras and their representations},
Lecture Notes in Mathematics, 933 (1982), 117-126. Berlin, Heidelberg,
New York: Springer.


\bibitem {Ka1} V. Kac, Superconformal algebras and transitive group actions on
quadrics, Commun. Math. Phys. 186, (1997) 233-252.

\bibitem {KL} V. Kac,  J. van de Leuer, {\it On classification of superconformal
algebras}, Strings 88, Sinapore: World Scientific, 1988.




\bibitem{KS} I. Kaplansky, L. J. Santharoubane, {\it Harish-Chandra modules over the Virasoro algebra}, Infinite-dimensional groups with applications (Berkeley, Calif. 1984), 217--231, Math. Sci. Res. Inst. Publ., 4, Springer, New York, 1985.


\bibitem{LG} G. Liu, X. Guo. {\it Harish-Chandra modules over generalized Heisenberg-Virasoro algebras}. Israel J. Math. 204 (2014), no. 1, 446-468.

\bibitem{LZ} G. Liu, K. Zhao. {\it Irreducible Harish Chandra modules over the derivation algebras of rational quantum tori}. Glasg. Math. J. 55 (2013), no. 3, 677-693.

\bibitem{LPX1} D. Liu, Y. Pei, and L. Xia, {\it Whittaker modules for the super-Virasoro algebras}, J. Algebra Appl.  18 (2019),  1950211.

\bibitem{LPX2} D. Liu, Y. Pei, and L. Xia, {\it Simple restricted modules for Neveu-Schwarz algebra}, J. Algebra, 546 (2020), 341-356.


\bibitem{Ma} O. Mathieu,  {\it Classification of Harish-Chandra
modules over the Virasoro Lie algebra}, Invent. Math. 107(1992), 225-234.


\bibitem{MZe} C. Mart\'{\i}nez, E. Zelmanov, {\it Graded modules over superconformal algebras}, Non-Associative and Non-Commutative Algebra and Operator Theory. Springer International Publishing, 2016, 41-53


\bibitem{MZ} V. Mazorchuk, K. Zhao, {\it Supports of weight modules over Witt algebras. Proc. Roy. Soc. Edinburgh Sect. A}, 141 (2011), no. 1, 155-170.




\bibitem{S2} Y. Su. {\it Classification of Harish-Chandra modules over the super-Virasoro algebras},  Commun. Alg.  23(10) (1995), 3653-3675.


\bibitem{S1}  Y. Su, {\it A classification of indecomposable $sl_2({\mathbb C})$-modules and a conjecture of Kac on irreducible modules over the Virasoro algebra}, J. Alg, 161(1993), 33-46.

\bibitem{LX} Y. Xue, R. L\"{u}. {\it Simple weight modules with finite dimensional weight spaces over Witt superalgebras},  arXiv:2001.04089.

\bibitem{XL} Y. Xue, R. L\"{u}. {\it Classification of simple bounded weight modules of the Lie algebra of vector fields on $\bC^n$}, , arXiv: 2001.04204.

\end{thebibliography}

\end{document}